\documentclass[letterpaper, 10 pt, conference]{ieeeconf}

\bstctlcite{bstctl:etal, bstctl:nodash, bstctl:simpurl}

\usepackage{anyfontsize}
\usepackage{array} 
\usepackage{stfloats}   
\usepackage{mathtools}
\usepackage{xcolor}
\usepackage[utf8]{inputenc}
\usepackage{tabularx}
\usepackage{booktabs} 
\usepackage{mathrsfs}
\usepackage{graphics} 

\usepackage{amsmath} 
\usepackage{amssymb}  
\usepackage{amsthm}
\usepackage{cite}
\usepackage{bm}
\usepackage{acronym}
\usepackage{paralist}
\usepackage{float}
\usepackage[dvipsnames]{xcolor}
\usepackage{epstopdf}
\usepackage{multicol}
\usepackage{tikz}
\usepackage{hyperref}
\usepackage{graphicx}
\usepackage{soul}
\usepackage{macros} 

\usepackage{stfloats}
\fnbelowfloat

 \newcommand{\mydarkred}{\color{black}}
 \newcommand{\myblack}{\color{Black}}

\makeatletter
\hypersetup{colorlinks=true}
\AtBeginDocument{\@ifpackageloaded{hyperref}
  {\def\@linkcolor{blue}
  \def\@anchorcolor{red}
  \def\@citecolor{red}
  \def\@filecolor{red}
  \def\@urlcolor{black}
  \def\@menucolor{red}
  \def\@pagecolor{red}
\begingroup
  \@makeother\`%
  \@makeother\=%
  \edef\x{%
    \edef\noexpand\x{%
      \endgroup
      \noexpand\toks@{%
        \catcode 96=\noexpand\the\catcode`\noexpand\`\relax
        \catcode 61=\noexpand\the\catcode`\noexpand\=\relax
      }%
    }%
    \noexpand\x
  }%
\x
\@makeother\`
\@makeother\=
}{}}
\makeatother

\DeclareMathOperator{\divergence}{\mathrm{div}}

\DeclareMathOperator{\grad}{\mathrm{grad}}
\DeclareMathOperator{\interior}{\mathrm{interior}}

\newcommand{\bequ}{\begin{eqnarray}}
\newcommand{\eequ}{\end{eqnarray}}

\newcommand{\differential}{\mathrm{d}}
\newcommand{\dist}{\mathrm{dist}}

\IEEEoverridecommandlockouts

\def\BibTeX{{\rm B\kern-.05em{\sc i\kern-.025em b}\kern-.08em
    T\kern-.1667em\lower.7ex\hbox{E}\kern-.125emX}}

\definecolor{AA}{RGB}{34,139,34}

\IEEEoverridecommandlockouts
\begin{document}
\bstctlcite{IEEE_b:BSTcontrol}
\title{\LARGE \bf 
{
Schrödinger Bridge Over A Compact Connected Lie Group}

\thanks{This research has been supported by NJIT's startup funds.}
\thanks{Hamza Mahmood$^*$ and Adeel Akhtar$^*$ are with the 
Department of Mechanical \& Industrial Engineering at the New Jersey Institute of Technology, Newark, NJ 07102, USA (email: \texttt{\{hm576, adeel.akhtar\}@njit.edu})}
\thanks{{Abhishek Halder$^{\dag}$ is with the Department of Aerospace Engineering at
Iowa State University, Ames, IA 50011, USA (email: \texttt{ahalder@iastate.edu})}}}

\author{Hamza Mahmood$^*$ \and Abhishek Halder$^{\dag}$ \and Adeel Akhtar$^*$}
\maketitle

\begin{abstract}
This work studies the Schrödinger bridge problem for the kinematic equation on a compact connected Lie group. The objective is to steer a controlled diffusion between given initial and terminal densities supported over the Lie group while minimizing the control effort. We develop a coordinate-free formulation of this stochastic optimal control problem that respects the underlying geometric structure of the Lie group, thereby avoiding limitations associated with local parameterizations or embeddings in Euclidean spaces. We establish the existence and uniqueness of solution to the corresponding Schrödinger system. Our results are constructive in that they derive a geometric controller that optimally interpolates probability densities supported over the Lie group. To illustrate the results, we provide numerical examples on $\mathsf{SO}(2)$ and $\mathsf{SO}(3)$. The codes and animations are publicly available at \texttt{\href{https://github.com/gradslab/SbpLieGroups.git}{https://github.com/gradslab/SbpLieGroups.git}}.
\end{abstract}

{\myblack
% ...    
}

\section{Introduction}\label{sec:Introduction} 
The Schrödinger bridge problem (SBP) has emerged \cite{chen2021stochastic,CalHal2022} as a powerful framework for controlling uncertainty in stochastic dynamical systems, as it enables minimum effort steering of joint state probability density functions (PDFs) under deadline constraint, instead of merely mitigating uncertainties. The existing SBP literature in systems-control — both theoretical~\cite{CheGeoPav2016, CalHal2019, CalHal2020, CalHal_RefSchBri2020,teter2025hopf,teter2025schrodinger} and application-oriented works~\cite{haddad2020prediction,nodozi2023neural,teter2025probabilistic} — has focused on Euclidean spaces. Here, we consider SBP for controlled diffusion on a compact connected Lie group, and develop a geometric solution for the same. 

Specifically, we consider the stochastic differential equation (SDE) corresponding to the kinematic equation on a compact connected Lie group $\mathsf{G}$, namely, the kinematic equation on $\mathsf{G}$ perturbed by Brownian noise. The controlled state PDFs are supported on $\mathsf{G}$, and Haar measure is used to define integration of functions, including PDFs, over $\mathsf{G}$. Using Hilbert's projective metric~\cite{Bir1957,Bus1973} defined on a closed solid cone in a Banach space, we prove that there exists a unique solution—up to a reciprocal scaling—to the Schrödinger system on $\mathsf{G}$. This solution yields the desired geometric (i.e., coordinate-free) controller for the SBP on $\mathsf{G}$. 

We emphasize that our control design is based on the intrinsic geometric structure of the Lie group, without embedding it in a Euclidean space. This geometric perspective is particularly relevant for Lie groups such as $\SO3$, which arise naturally in applications such as attitude control \cite{lee2013stochastic}, DNA statistical mechanics, steering of flexible needles, and mobile robots~\cite[Ch. 20]{chirikjian2011}.

Our novel contributions are to show the following:
\begin{enumerate}
    \item if $\mathsf{G}$ is a compact connected Lie group, then in the real Banach space $C(\mathsf{G}) $ with the supremum norm\footnote{Here $C(\mathsf{G})$ denotes the space \cite{Piotr2007} of real-valued continuous functions on $\mathsf{G}$, and is equipped with the norm $\|f\|_{\infty} \eqdef \sup_{R \in \mathsf{G}} |f(R)|$ for any $f \in C(\mathsf{G})$.}, the interior of the set $C_{\geq 0}(\mathsf{G})\eqdef\{f \in C(\mathsf{G}) \mid \text{ for all } R \in \mathsf{G}, f(R) \geq 0\}$ is the set $C_{+}(\mathsf{G})\eqdef\{f \in C(\mathsf{G}) \mid \text{ for all } R \in \mathsf{G}, f(R) > 0\}$ (Lemma~\ref{lem:intC{geq 0}}),
    \item using the above, we show that there exists a unique solution (in a projectivized sense) to the Schrödinger system on the compact connected Lie group $\mathsf{G}$ (Theorem~\ref{thm:Soln_Schrödinger_sys_G}), and
    \item from above, we compute and exemplify geometric (coordinate-free) solution to the SBP (Problem~\ref{problem:SBP_on_G}) on a compact connected Lie group $\mathsf{G}$.
\end{enumerate}
In Sec.~\ref{sec:Preliminaries}, we introduce the key ideas and definitions required for problem formulation. Sec.~\ref{sec:Problem_formulation} formulates the problem. Sec.~\ref{sec:schrodinger-bridge-geom} then presents a geometric (coordinate-free) solution to the problem after proving the main result. Sec.~\ref{sec:numerical_simulations} illustrates the results via numerical simulations. Concluding remarks are given in Sec.~\ref{sec:Conclusion}. All proofs appear in the Appendix.

\section{Preliminaries}
\label{sec:Preliminaries}
\noindent\textbf{Lie Group.} A \emph{Lie group} $\mathsf{G}$ is a topological group that is also a smooth manifold, and in which the group and the 
inverse operations are smooth. We say a Lie group is \emph{compact} (resp. \emph{connected}) if the underlying topological space is compact (resp. connected). 
Familiar examples of Lie groups include the matrix Lie groups~\cite{tapp2016matrix} such as orthogonal and special orthogonal groups: $\mathsf{O}(n)$ and $\mathsf{SO}(n)$, unitary and special unitary groups: $\mathsf{U}(n)$ and $\mathsf{SU}(n)$, for integer $n\geq 2$. All of these four Lie groups are compact, all but $\mathsf{O}(n)$ are connected.

\noindent\textbf{Calculus on Lie Group.} Let $\mathfrak{X}(\mathsf{G})$ denote the set of all smooth vector fields on $\mathsf{G}$. Then the \emph{Laplace--Beltrami operator on $\mathsf{G}$} is the mapping {\mydarkred $\Delta_\mathsf{G} : C^{\infty}(\mathsf{G}) \to C^{\infty}(\mathsf{G})$}, given by
\(
\Delta_\mathsf{G} f \;\eqdef\; \divergence ( \grad f ), \; f \in C^{\infty}(\mathsf{G}), \)
where $\grad : C^{\infty}(\mathsf{G}) \to \mathfrak{X}(\mathsf{G})$ is the gradient of a smooth real-valued function defined on $\mathsf{G}$, and $\divergence : \mathfrak{X}(\mathsf{G}) \to C^{\infty}(\mathsf{G})$ is the divergence of a smooth vector field on $\mathsf{G}$.

To integrate functions over a compact Lie group, we require the notion of \emph{Haar measure}, which is an example of Radon measure. For our compact Lie group $\mathsf{G}$, a \emph{Radon measure} is a measure on the $\sigma$-algebra of Borel sets of $\mathsf{G}$ that is finite on all compact subsets, outer regular on all Borel subsets, and inner regular on all open subsets of $\mathsf{G}$~\cite[p. 212]{folland1999real}. A \emph{left Haar measure} (or simply \emph{Haar measure}) is a (non-zero) left-invariant Radon measure $\mu$ on $\mathsf{G}$. Left-invariant means that for any Borel subset $E \subset \mathsf{G}$ and for any $g \in \mathsf{G}$, we have $\mu(gE)=\mu(E),$ where $gE \eqdef \{\, gx \, | \, x \in E \,\}$. Likewise, we can define a right Haar measure. 

On any locally compact Lie group $\mathsf{G}$, there exists a left Haar measure~\cite[Theorem 5.1.1]{Faraut2008} that is unique up to multiplication by a positive scalar. For compact $\mathsf{G}$, this left Haar measure can be normalized to be a probability measure.

\noindent\textbf{The Kinematic Equation on a Lie Group.} Let $\mathsf{G}$ be a Lie group with identity element $e$, and let\footnote{Here $\mathsf{T}_{e}\mathsf{G}$ denotes the tangent space of $\mathsf{G}$ at the identity element $e$.} $\mathfrak{g}:={\mydarkred \mathsf{T}_{e}\mathsf{G}}$ be the associated {\mydarkred finite-dimensional} Lie algebra having dimension $\dim \mathfrak{g}$. The ``hat map" $\hat{\cdot} : \mathbb{R}^{\dim \mathfrak{g}} \to \mathfrak{g}$ is an isomorphism between the vector spaces $\mathbb{R}^{\dim \mathfrak{g}}$ and $\mathfrak{g}$. The inverse of the hat map is $(\cdot)^{\vee} : \mathfrak{g} \to \mathbb{R}^{\dim \mathfrak{g}}$.

We denote the configuration of a system evolving on $\mathsf{G}$ at time $t$ by $R(t) \in \mathsf{G}$. Its time derivative $\dot{R}(t)$ lies in the tangent space ${\mydarkred \mathsf{T}_{R(t)}\mathsf{G}}$. Let $\Omega(R,t)$ be a vector in $\mathbb{R}^{\dim \mathfrak{g}}$ at time $t$. The time evolution of the configuration is then governed by the \emph{kinematic equation}~\cite[Ch. 5]{bullo2005geometric}:
\begin{equation}
\label{eq:kinematic_eq_on_G}
\dot{R} = R \, \hat{\Omega}(R,t), \quad\text{where}\quad\dot{R}(t):=\frac{\differential}{\differential t}R(t).
\end{equation}
Since $\hat{\Omega}(R,t) \in \mathfrak{g} = {\mydarkred \mathsf{T}_{e}\mathsf{G}}$, left multiplication by $R(t)$ maps this element of the Lie algebra to an element in the tangent space $\mathsf{T}_{R(t)}\mathsf{G}$. So $R\hat{\Omega}(R,t)\in\mathsf{T}_{R(t)}\mathsf{G}$ ensures that the trajectory $R(t)$ remains on $\mathsf{G}$ at all times $t$.

The kinematic equation \eqref{eq:kinematic_eq_on_G} is a deterministic differential equation. Next, we describe stochastic differential equation (SDE) on $\mathsf{G}$.

\noindent\textbf{Stratonovich SDE on a Compact Lie Group.}
We consider the \emph{Stratonovich SDE} on a compact Lie group $\mathsf{G}$ with its Lie algebra $\mathfrak g$, wherein $\mathsf{G}$ is viewed as a Riemannian manifold~\cite[Ch. 1]{hsu2002}.

Specifically, for fixed $m \in \mathbb N$, consider an $\mathbb R^m$-valued standard Brownian motion $(W_t)_{t\ge0}$. Given smooth vector fields $V_i$ on $\mathsf{G}$ ${\mydarkred \forall i \in \{0, 1,\dots,m\}}$, we say a $\mathsf{G}$-valued continuous semimartingale $(R(t))_{t\ge 0}$ solves the Stratonovich SDE
\[
\differential R \;=\; V_0(R)\:\differential t \;+\; \sum_{i=1}^m V_i(R)\circ \differential W_t^i, 
\qquad R(0) = {\mydarkred g} \in \mathsf{G},
\]
if for any smooth test function $f\in C^\infty(\mathsf{G})$, the following holds:  
\begin{align*}
f(R(t)) = f(R(0)) &+ \!\int_0^t \!(V_0 f)(R(s))\:\differential s\\
&+\sum_{i=1}^m \!\int_0^t \!(V_i f)(R(s))\circ\differential W_s^i.
\end{align*} 
In above, ${\mydarkred V_i f}$ denotes the directional derivative of $f$ along the vector field ${\mydarkred V_i}$. The Stratonovich SDE used throughout this work is consistent with the classical chain rule of differential geometry~\cite[Ch. 1]{hsu2002}.
\begin{definition}[Stratonovich SDE for the Kinematics on a Compact Lie Group]\label{def:StratoSDE}
Consider a compact Lie group $\mathsf{G}$ with its Lie algebra $\mathfrak g$. Let $\Omega: \mathsf{G} \times[0,1]\to\mathbb{R}^{\dim \mathfrak{g}}$ be a time-varying vector field on $\mathsf{G}$. For fixed $m \in \mathbb N$, consider an $\mathbb R^m$-valued standard Brownian motion $(W_t)_{t\ge0}$, and let $\{\sigma_i : \mathsf{G} \to \Real^{\dim \mathfrak{g}}\}_{i=1}^m$ be smooth noise amplitudes. Then the stochastic kinematic equation on $\mathsf{G}$ is the \emph{Stratonovich SDE}
$$\differential R \;=\; R\,\widehat{\Omega}(R,t)\:\differential t \;+\; \sum_{i=1}^m R\,\widehat{\sigma_i}(R)\;\circ \differential W_t^i, \quad R(0)\in \mathsf{G}.$$
\end{definition}

\begin{remark}[Unimodularity and compact Lie groups]
A locally compact Lie group where left Haar measure is also a right Haar measure is called \emph{unimodular}. 
Every compact Lie group is unimodular~\cite[Proposition 5.1.3]{Faraut2008}. Since our subsequent development considers compact Lie groups, we simply refer to the measure as Haar measure.
\end{remark}

\begin{proposition}[Brownian Motion on a Unimodular Lie Group]\label{Prop:unimodular}~\cite[Corollary 1]{Lee2025}
Let $\mathsf{G}$ be a unimodular Lie group with its Lie algebra $\mathfrak{g}$, and $n:=\dim\mathfrak{g}$. Consider an inner product $\langle \cdot, \cdot \rangle_{\mathfrak{g}} : \mathfrak{g} \times \mathfrak{g} \to \mathbb{R}$ and an orthonormal basis $\{e_1^{\mathfrak{g}}, e_2^{\mathfrak{g}}, \dots , e_n^{\mathfrak{g}}\}$ of $\mathfrak{g}$ w.r.t. $\langle \cdot, \cdot \rangle_{\mathfrak{g}}$\,. Let $(W_t)_{t\ge0}$ be an $\mathbb R^n$-valued standard Brownian motion. Then the Brownian motion on $\mathsf{G}$ is given by  
\begin{equation}
\label{eq:SDE_Brownian_unimodular} 
R^{-1}\:\differential R \;=\; \sum_{i=1}^{n} e_i^{\mathfrak{g}} \circ\differential W^i_t .
\end{equation}
\end{proposition}
\noindent\textbf{Hilbert's Projective Metric.} We need the following definitions to define this metric.
\begin{definition}[Closed solid cone in a Banach space {\cite[Sec.~2]{Bus1973}}]
\label{def:closed_solid_cone} Let $\left(\mathcal{X}{\mydarkred, \| \cdot \|_{\mathcal{X}}}\right)$ be a real Banach space. We say $\mathcal{K}\subseteq\mathcal{X}$ is a \emph{closed solid cone} if $\mathcal{K}$ is closed in $\mathcal{X}$, and
\begin{enumerate}[(i)]
    \item $\interior \mathcal{K}$ is non-empty,
    \item for any $\mathrm{\mathbf{x}},\mathrm{\mathbf{y}} \in \mathcal{K}$, we have $\mathrm{\mathbf{x}} + \mathrm{\mathbf{y}} \in \mathcal{K}$,   
    \item for any $\lambda \geq 0$ and $\mathrm{\mathbf{x}}\in \mathcal{K}$, we have $\lambda \mathrm{\mathbf{x}} \in \mathcal{K}$, and   
    \item $\mathcal{K} \cap -\mathcal{K} = \{\mathbf{0}\},$ where $\mathbf{0}$ is the zero element of $\mathcal{X}$. 
\end{enumerate}
\end{definition} 
\begin{definition}[Hilbert's projective metric {\cite[Def.~2.1 and Def.~2.2]{Bus1973}}]
\label{def:Hilbert_metric}
Consider a closed solid cone $\mathcal{K}$ in a real Banach space $\left(\mathcal{X}{\mydarkred, \| \cdot \|_{\mathcal{X}}}\right)$. 
Let $\preceq$ be a partial order relation in $\mathcal{X}$ induced by $\mathcal{K}$, defined as follows: for any $\mathrm{\mathbf{x}},\mathrm{\mathbf{y}} \in \mathcal{X}$, 
$
\mathrm{\mathbf{x}} \preceq \mathrm{\mathbf{y}} \,\Leftrightarrow\, \mathrm{\mathbf{y}} - \mathrm{\mathbf{x}} \in \mathcal{K}\,.     
$
Define $\mathcal{K}^{+} \eqdef \mathcal{K} \setminus \{\mathbf{0}\}$ and let $M(\mathbf{x}, \mathbf{y}) \eqdef \inf \left\{ \lambda \mid \mathbf{x} \preceq \lambda \mathbf{y} \right\}$ and $m(\mathbf{x}, \mathbf{y}) \eqdef \sup \left\{ \lambda \mid \lambda \mathbf{y} \preceq \mathbf{x} \right\}$ $\forall\mathrm{\mathbf{x}}, \mathrm{\mathbf{y}} \in \mathcal{K}^{+}$, with the convention $\inf \emptyset = \infty$. Then \emph{Hilbert's projective metric} 
\begin{align}
\label{eq:Hilbert_proj_metric}
d_H(\mathbf{x}, \mathbf{y}) := \log\left(M(\mathbf{x}, \mathbf{y})/m(\mathbf{x}, \mathbf{y})\right) \quad \forall\mathrm{\mathbf{x}}, \mathrm{\mathbf{y}} \in \mathcal{K}^{+}.
\end{align}   
\end{definition}

 


\begin{remark}
\label{rem:d_H_properties}
From Definition~\ref{def:Hilbert_metric}, it follows that for any $\mathbf{x}, \mathbf{y} \in \interior \mathcal{K}$, $d_H(\mathbf{x},\mathbf{y})$ is finite \cite[proof of Theorem 2.1]{Bus1973}, and that $d_H(\lambda\mathbf{x},\mu\mathbf{y}) = d_H(\mathbf{x},\mathbf{y})$ $\forall\lambda, \mu > 0$ \cite[Lemma 2.2]{Bus1973}.
\end{remark}

\section{Problem Formulation}
\label{sec:Problem_formulation}
{\mydarkred In this section, we formulate the SBP for the kinematic equation on a compact connected Lie group $\mathsf{G}$ with its Lie algebra $\mathfrak{g}$ of dimension $n$. Despite apparent similarity to the SBP on $\Real^n$, the subsequent analysis will reveal that new and nontrivial issues arise due to the geometric objects involved.

To begin, let $\mu$ be the Haar (volume) measure on $\mathsf{G}$, and define $\mathcal{U}$ to be the set of finite energy controls on $\mathsf{G}$, i.e.,
\begin{equation}
\label{eq:Controls}
\begin{aligned}
    \mathcal{U} \eqdef 
    \biggl\{ {\Omega} : \mathsf{G} \times [0,1] &\rightarrow \mathbb{R}^{n} \;\bigg|\;
    \text{for each } t\in[0,1], \\[2pt]
    &\int_{\mathsf{G}} \|\Omega(R,t)\|_{2}^2\, \differential\mu(R) < \infty
    \biggr\}.
\end{aligned}
\end{equation}
Also, let
\begin{align}
\label{eq:P_2_G}
&\mathcal{P}_2(\mathsf{G}) \eqdef \biggl\{ \rho : \mathsf{G} \rightarrow \mathbb{R}_{\geq 0} \bigg|
\int_{\mathsf{G}} \rho \, \mathrm{d}\mu(R) = 1, \notag \\ \;
&\;\;\;\;\;\;\;\;\;\;\;\;\;\;\;\;\;\int_{\mathsf{G}} {\dist(R,R_0)^2} \, \rho \, \mathrm{d}\mu(R) < \infty \biggr\},
\end{align} 
where $\dist(R,R_0)$ denotes the \emph{Riemannian distance} of $R \in \mathsf{G}$ from a fixed $R_0\in \mathsf{G}$.



}
\begin{problem}
\label{problem:SBP_on_G}
{\mydarkred Let $\mathsf{G}$ be a compact connected Lie group with its Lie algebra $\mathfrak{g}$ of dimension $n$. Consider an inner product $\langle \cdot, \cdot \rangle_{\mathfrak{g}} : \mathfrak{g} \times \mathfrak{g} \to \mathbb{R}$\, and an orthonormal basis $\{e_1^{\mathfrak{g}}, e_2^{\mathfrak{g}}, \dots , e_n^{\mathfrak{g}}\}$ of $\mathfrak{g}$ w.r.t. $\langle \cdot, \cdot \rangle_{\mathfrak{g}}$\,. Let $(W_t)_{t\ge0}$ be an $\mathbb R^n$-valued standard Brownian motion. 
Design a controller $\Omega \in \mathcal{U}$ that solves the stochastic optimal control problem:
\begin{subequations}
\label{eq:SBP_on_G}
\begin{align}
\inf_{\Omega \in \mathcal{U}}& \;\mathbb{E} \left\{ \int_{0}^{1} \frac{1}{2} \|\Omega(R,t)\|_{2}^{2} \, \mathrm{d}t \right\} \label{1st_SBP_on_G} \\
\mathrm{subject\,\,to}& \; 
\mathrm{d}R = R\, \widehat{\Omega}(R,t)\, \mathrm{d}t + \sigma\sum_{i=1}^{{n}} {R e_i^{\mathfrak{g}} \; \circ \; }     \mathrm{d}W^i_t, \label{2nd_SBP_on_G} \\
&R(t=0) \sim \rho_{0}(R), \quad R(t=1) \sim \rho_{1}(R){\mydarkred ,} \label{3rd_SBP_on_G}
\end{align}
\end{subequations} 
where $\sigma > 0$ denotes the (isotropic) diffusion strength, $\rho_0,\rho_1 \in \mathcal{P}_2(\mathsf{G})$, and the expectation in \eqref{1st_SBP_on_G} is w.r.t. the controlled state PDF $\rho$\,, that is, $\mathbb{E}\{\cdot\} := \int_{\mathsf{G}} (\cdot) \, \rho \, \mathrm{d}\mu(R)$\,.} 
\end{problem}%
\noindent An equivalent variational formulation for Problem \ref{problem:SBP_on_G} is
\begin{subequations}
\label{eq:SBP_on_G_varprob}
\begin{align}
\inf_{\rho, \Omega} \quad \quad \; & \int_0^1 \!\!\int_{\mathsf{G}} \frac{1}{2} \|\Omega(R,t)\|_{2}^2 \, \rho(R,t) \, \mathrm{d}\mu(R) \, \mathrm{d}t \label{1st_SBP_on_G_varprob} \\
\text{subject to} \; &\partial_t \rho = -\mathrm{div}(\rho R\:\widehat{\Omega}(R,t)) + \frac{\sigma^2}{2} \Delta_{\mathsf{G}} \rho(R,t),
\label{2nd_SBP_on_G_varprob} \\
&\rho(R,0) = \rho_0(R), \qquad \rho(R,1) = \rho_1(R), \label{3rd_SBP_on_G_varprob}
\end{align}
\end{subequations}
where \eqref{2nd_SBP_on_G_varprob} is the Fokker-Planck or Kolmogorov’s forward PDE associated with the controlled SDE\footnote{The SDE \eqref{2nd_SBP_on_G} is a consequence of Definition \ref{def:StratoSDE} and Proposition \ref{Prop:unimodular}.} \eqref{2nd_SBP_on_G} on $\mathsf{G}$.

\section{A Geometric Coordinate-free Solution}
\label{sec:schrodinger-bridge-geom}
Using the notations and ideas introduced thus far, this section establishes the existence-uniqueness of a solution as well as an optimal controller synthesis for Problem~\ref{problem:SBP_on_G}.

{
{  
\subsection{Preparatory results}
For the differential operator $L :=\tfrac{\sigma^2}{2}\,\Delta_{\mathsf{G}}$, consider the heat semigroup \((T_t)_{0 \leq t \leq 1}= (e^{t L})_{0 \leq t \leq 1}\) on \(L^2(\mathsf{G},\mu)\) and the heat kernel \(k_t:\mathsf{G} \times \mathsf{G} \to\mathbb R_{\geq 0}\), i.e., for any $f\in L^2(\mathsf{G})$ and for each $t \in [0,1]$,
\begin{equation}
\label{eq:(T_t f)(R)}
(T_t f)(R) \;=\; \int_{\mathsf{G}} k_t(R,\tilde{R})\,f(\tilde{R})\,\differential\mu(\tilde{R}).
\end{equation} 
Because of the connectedness of $\mathsf{G}$, the heat kernel $k_t$ is strictly positive~\cite[Corollary 8.12]{Grig2012}. 

Recall the definitions of the sets $C_{\geq 0}(\mathsf{G}), C_{+}(\mathsf{G})$ from Sec. \ref{sec:Introduction}. 
Notice that $C_{+}(\mathsf{G})$ is nonempty: the function $u : \mathsf{G} \to \mathbb{R}$ defined by $u(R) = 1$ for all $R \in \mathsf{G}$, is clearly in $C_{+}(\mathsf{G})$\,. 
We next present a few auxiliary lemmas that find use in proving our main result, Theorem \ref{thm:Soln_Schrödinger_sys_G}.

\begin{lemma}[Interior of \(C_{\geq 0}(\mathsf{G})\)]
\label{lem:intC{geq 0}}
{\mydarkred Let $\mathsf{G}$ be a compact connected Lie group.} Then, in the real Banach space $\left(C(\mathsf{G}), \|\cdot\|_{\infty} \right)$, we have 
$\interior C_{\geq 0}(\mathsf{G}) = C_{+}(\mathsf{G})$.
\end{lemma}
\noindent With Lemma~\ref{lem:intC{geq 0}}, it follows from Definition~\ref{def:closed_solid_cone} that $C_{\geq 0}(\mathsf{G})$ is a closed solid cone in the Banach space $\left(C(\mathsf{G}), \|\cdot\|_{\infty} \right)$. 
}

{\myblack
\begin{lemma}\cite[Theorems 2.1,\, 4.1,\, 4.2,\, 4.4]{Bus1973}\label{lem:Complete_metric_space}
Let $\mathsf{G}$ be a compact connected Lie group, and consider the Banach space $\bigl(C(\mathsf{G}),\|\cdot\|_{\infty}\bigr)$ with unit sphere $
        U \eqdef \{\, f \in C(\mathsf{G}) \,\,| \,\, \| f\|_{\infty} = 1\,\}$. Then the metric space $\bigl(C_{+}(\mathsf{G})\cap U,\; d_H\bigr)$ is complete.
\end{lemma}

{\mydarkred \begin{lemma}
\label{lem:R_f_isometry}
Let $\mathsf{G}$ be a compact connected Lie group and for any $f\in C_{+}(\mathsf{G})$, define $R_f:C_{+}(\mathsf{G})\to C_{+}(\mathsf{G})$ as
\((R_f g)(x)=f(x)/g(x), \quad x\in \mathsf{G}.\)
Then $R_f$ is an isometry w.r.t. Hilbert's projective metric $d_H$, i.e.,  $d_H(R_f(g_1), R_f(g_2)) = d_H(g_1, g_2)$ $\forall g_1, g_2 \in C_{+}(\mathsf{G})$.  
\end{lemma}}

\begin{remark}
\label{rem:T_1_Dom&Codom_C+}
Since the heat kernel $k_t$ is strictly positive on $\mathsf{G} \times \mathsf{G}$, it follows from \eqref{eq:(T_t f)(R)} that for any $f \in C_{+}(\mathsf{G})$, we have $(T_1 f)(R) > 0$ $\forall R\in \mathsf{G}$. Also, for any $f \in L^{2}(\mathsf{G})$ and $t>0$, we have $T_t f \in C^{\infty}(\mathsf{G})$ ~\cite[Theorem 7.6]{Grig2012}. Hence, for any $f \in C_{+}(\mathsf{G})$, we get $T_1 f \in C_{+}(\mathsf{G})$.
\end{remark}%
\begin{lemma}\cite[Proof similar to that of Lemma 5]{Che2016}
\label{lem:T_1_str_contraction}
Let $\mathsf{G}$ be a compact connected Lie group. There exists a constant $c \in [0,1)$ such that for all $f_1, f_2 \in C_{+}(\mathsf{G})$\,, we have $
d_H(T_1 f_1, T_1 f_2) \leq c \,\, d_H(f_1, f_2).
$
\end{lemma}

} 
\noindent{\myblack
Next, we investigate the necessary conditions for optimality, leading to the main result. 
\subsection{Conditions for optimality and the Schrödinger system}
The Lagrangian associated with \eqref{eq:SBP_on_G_varprob} is
\begin{align}
\label{eq:Lagrangian}
\mathcal{L}(\rho,\Omega,S)=\!\!\int_0^1\!\!\int_{\mathsf{G}} \Bigl\{ &\tfrac{1}{2}{\mydarkred \|\Omega(R,t)\|_{2}^2}\:\rho
\:+\:S\Bigl(\partial_t \rho  +\operatorname{div}(\rho R \widehat{\Omega}) \notag \\ 
&- \tfrac{\sigma^2}{2}\Delta_{\mathsf{G}} \rho\Bigr)\Bigr\} \differential\mu \,\differential t,
\end{align}
where $S : \mathsf{G} \times [0,1] \to \mathbb{R}$ is a smooth Lagrange multiplier. Let $\mathcal{P}_{01}(\mathsf{G}) \eqdef \bigl\{ \rho: \mathsf{G} \times [0,1] \to \mathbb{R}_{\geq 0} \,\big|\, \rho(\cdot,0) = \rho_0, \rho(\cdot,1) = \rho_1,  \;\text{for each }t \in [0,1],\int_{\mathsf{G}} \rho(R,t)\:\differential\mu(R) = 1\bigr\}$.
Performing the unconstrained minimization of the Lagrangian $\mathcal{L}$ over $\mathcal{P}_{01}(\mathsf{G}) \times \mathcal{U}$, we find that the optimal pair $(\rho^{\mathrm{opt}}, \Omega^{\mathrm{opt}})$ for \eqref{eq:SBP_on_G_varprob} must solve the following system of coupled PDEs:
\begin{subequations}
\label{eq:rhoOpt&psi_G}
\begin{align}
&\partial_t S + \frac{1}{2}{\mydarkred \norm{\left(R^{-1}\grad S\right)^{\vee}}_2^2} = -\,\frac{\sigma^2}{2}\,\Delta_{\mathsf{G}}S, 
\label{eq:1st_rhoOpt&psi_G} \\ 
&\partial_t\rho^{\mathrm{opt}} + \operatorname{div}(\rho^{\mathrm{opt}}\grad S) = \frac{\sigma^2}{2}\,\Delta_{\mathsf{G}}\rho^{\mathrm{opt}}, 
\label{eq:2nd_rhoOpt&psi_G} \\
&\rho^{\mathrm{opt}}(\cdot,0)=\rho_0,\quad \rho^{\mathrm{opt}}(\cdot,1)=\rho_1,
\label{eq:3rd_rhoOpt&psi_G}
\end{align}
\end{subequations}
and the optimal control 
\begin{equation}
\label{eq:Omega^{opt}} 
\Omega^{\mathrm{opt}}(R,t) = \left(R^{-1}\grad S\left(R,t\right)\right)^{\vee}.
\end{equation}
Using the Hopf--Cole transform $(\rho^{\text{opt}}, S) \mapsto (\varphi, \widehat{\varphi})$ given by 
\begin{subequations}
\label{eq:Hopf-Cole}
\begin{align}
&\varphi(R,t)=\exp\left(\frac{S(R,t)} {\sigma^2}\right),\label{1st_Hopf-Cole}\\ 
&\widehat{\varphi}(R,t) = \rho^{\mathrm{opt}}(R,t) \exp\left(-\frac{S(R,t)}{\sigma^2}\right),
\label{2nd_Hopf-Cole}
\end{align}
\end{subequations}
we can transform the system \eqref{eq:rhoOpt&psi_G} into the pair of linear (forward and backward) heat equations
\begin{align}
\partial_t\varphi = -\frac{\sigma^2}{2}\,\Delta_{\mathsf{G}}\varphi, \quad \partial_t\widehat{\varphi} = \frac{\sigma^2}{2}\,\Delta_{\mathsf{G}}\widehat{\varphi},
\label{HeatPDEsForwardBackward}
\end{align}
with coupled boundary conditions 
\begin{align}
\varphi(\cdot,0)\,\widehat{\varphi}(\cdot,0)=\rho_0,
\quad 
\varphi(\cdot,1)\,\widehat{\varphi}(\cdot,1)=\rho_1\,. 
\label{eq:coupled_bnd_cond_G}
\end{align}
Letting 
\(\varphi_1 \eqdef \varphi(\cdot, t=1),\widehat{\varphi}_0 \eqdef \widehat{\varphi}(\cdot, t=0)\), we can express the solution of \eqref{HeatPDEsForwardBackward} on $\mathsf{G}$ using the heat semigroup \eqref{eq:(T_t f)(R)} as 
\begin{align}
\varphi(R,t) := (T_{1-t}\varphi_1)(R), \quad \widehat{\varphi}(R,t) := (T_t \widehat{\varphi}_0)(R)\,, 
\label{eq:soln_heat_eqns_G}
\end{align}
$\forall R \in \mathsf{G},\;t\in[0,1]$.

Combining \eqref{eq:coupled_bnd_cond_G} and \eqref{eq:soln_heat_eqns_G}, we arrive at the so-called \emph{Schrödinger system} on $\mathsf{G}$:  
\begin{align}
\rho_0=\widehat{\varphi}_0\,T_{1}\varphi_1\,,\quad \rho_1=\varphi_1\, T_1\widehat{\varphi}_0,
\label{eq:Schrödinger_sys_G}
\end{align}
which is a pair of coupled integral equations in unknown function pair $(\varphi_1,\widehat{\varphi}_0)$.

We next prove, as our main result of this work, the existence-uniqueness of solution $(\varphi_1,\widehat{\varphi}_0)$ to this Schrödinger system \eqref{eq:Schrödinger_sys_G}. For doing so, we restrict the operator \(T_1\) to the subset $C_{+}(\mathsf{G})$ of $L^{2}(\mathsf{G})$, and take $C_{+}(\mathsf{G})$ as the codomain of $T_1$, which is valid because of Remark~\ref{rem:T_1_Dom&Codom_C+}. 
}

\begin{theorem}[Solution of the Schrödinger system on a compact connected Lie group]\label{thm:Soln_Schrödinger_sys_G}
Consider the SBP given as Problem \ref{problem:SBP_on_G} with endpoint PDFs $\rho_0 \,,\, \rho_1 \in C_+(\mathsf{G})$. Consider the operator $T_1 : C_{+}(\mathsf{G}) \to C_{+}(\mathsf{G})$\ with the heat kernel $k_1 : \mathsf{G} \times \mathsf{G} \to \Real_{> 0}$. Then there exist \(\varphi_1\,, \widehat{\varphi}_0 \,\in C_+(\mathsf{G})\) solving the Schrödinger system \eqref{eq:Schrödinger_sys_G}. Moreover, this pair $(\varphi_1\,, \widehat{\varphi}_0)$ is unique up to  reciprocal scaling: $(\varphi_1\,, \widehat{\varphi}_0) \to (\alpha\varphi_1\,, \widehat{\varphi}_0/\alpha)$ for any $\alpha>0$. 
\end{theorem}

\begin{corollary}\label{corollary:dynamicSinkhorn}
From the proof of Theorem~\ref{thm:Soln_Schrödinger_sys_G}, the dynamic Sinkhorn recursion $\varphi_{1} \rightarrow(\varphi_1)_{\text{next}}$ defined via \eqref{eq:varphi_eq} with arbitrary initial guess is guaranteed to converge to a fixed point with worst-case linear rate of convergence w.r.t. $d_{H}$.
\end{corollary}%
\noindent Corollary \ref{corollary:dynamicSinkhorn} provides the pair $(\varphi_1,\widehat{\varphi}_0)$ solving \eqref{eq:Schrödinger_sys_G} on $\mathsf{G}$. Using \eqref{eq:soln_heat_eqns_G}, we can then compute $\left(\varphi,\widehat{\varphi}\right)$, and thus the original primal-dual pair $(\rho^{\mathrm{opt}},S)$ from \eqref{eq:Hopf-Cole} as
\begin{align}
\!(\rho^{\mathrm{opt}},S)= \left(\varphi\widehat{\varphi},\sigma^2\log \varphi\right).
\label{RecoverPrimalDualPair}    
\end{align}
Notice that the optimally controlled joint PDF is unique since the product $\varphi\widehat{\varphi}$ cancels the initialization-dependent scaling constant. Using \eqref{eq:Omega^{opt}}, we find the optimal control
\begin{equation}
\label{eq:Omega^{opt}_soln}
\Omega^{\mathrm{opt}}(R,t)=\left(\sigma^2 R(t)^{-1}\grad \log \varphi(R,t)\right)^{\vee}.
\end{equation}
Thus, the controller $\Omega^{\mathrm{opt}}$ provides a geometric (coordinate-free) solution of the SBP on the compact connected Lie group $\mathsf{G}$, thereby solving Problem~\ref{problem:SBP_on_G}.

\section{Numerical Results}
\label{sec:numerical_simulations}
We next present numerical results for the solution of Problem \ref{problem:SBP_on_G} for $\SO2$ and $\SO3$. Our computation used log-domain stabilized \cite[Sec. 4.4]{peyre2019computational} dynamic Sinkhorn recursion (Corollary \ref{corollary:dynamicSinkhorn}). More examples and animations can be found in our 
\texttt{GitHub} repository\footnote{\label{footnote:code_link}
Code and animation at \href{https://github.com/gradslab/SbpLieGroups.git}{https://github.com/gradslab/SbpLieGroups.git}}.

\subsection{SBP on \texorpdfstring{$\mathsf{SO}(2)$}{SO(2)}}\label{subsec:SO2numericalexample} 
In $\SO2$, we have the Fourier series representation of the heat kernel \cite[Ch. 4.4]{stein2011fourier} $k_t(\theta_1,\theta_2) = \frac{1}{2\pi}\sum_{m\in\mathbb{Z}} e^{-\frac{1}{2}\sigma^2 m^2 t} e^{{\mathrm{i}}m\theta}$, ${\mathrm{i}}:=\sqrt{-1}$, $\theta:=\theta_1-\theta_2$,
where the diffusion strength $\sigma>0$. This allows us to numerically implement \eqref{eq:(T_t f)(R)} 
 via Fast Fourier Transform (FFT), and thereby the dynamic Sinkhorn recursion in Corollary \ref{corollary:dynamicSinkhorn}.

Since $\SO2\cong\mathbb{S}^{1}$ (circle group), we consider an instance of Problem \ref{problem:SBP_on_G} with endpoint von Mises PDFs
\begin{align}
\rho_i(\theta) = e^{\kappa \cos(\theta - \theta_{0i})}/(2\pi I_0(\kappa)), \,i\in\{0,1\},\,\kappa = 40, 
\label{eq:vonMises_dist}
\end{align}
where $I_0(\cdot)$ is the modified Bessel function of the first kind of order zero, $\theta_{00} = \pi/6$, $\theta_{01}=11\pi/6$. Fig.~\ref{fig:SBP-nearby-peaks} shows the resulting optimally controlled joint PDFs $\rho^{\mathrm{opt}}(\cdot,t)$, $t\in[0,1]$, wherein $\rho_0$ is shown in \emph{green solid}, and $ \rho_1$ is shown in \emph{dashed red}. This figure highlights that our SBP solution indeed evolves on $\SO2$ and not on an interval of length $2\pi$ of the real line. If the solution were on $[0,2\pi)$ with the corresponding optimal controller designed on the Euclidean space, then the $\rho^{\mathrm{opt}}(\cdot,t)$ would have traveled across the longer arc between the peaks, making the solution suboptimal. Instead, they move across the shorter arc, confirming the intrinsic geometry-preserving nature of the solution.     

\subsection{SBP on \texorpdfstring{$\mathsf{SO}(3)$}{SO(3)}}\label{subsec:SO3numericalexample}
The Lie algebra of $\SO3$ is $\mathfrak{so}(3)$, the set of $3\times 3$ skew-symmetric matrices. From Rodrigues’ formula, for any \(R \in \SO3\), there exists
\(\omega \in \mathbb{R}^{3}\) with \(\|\omega\|_2 \le \pi\) such that
\[
R = \boldsymbol{\exp}(\hat{\omega})
  = I
    + \frac{\sin\|\omega\|_2}{\|\omega\|_2}\,\hat{\omega}
    + \frac{1 - \cos\|\omega\|_2}{\|\omega\|_2^{2}}\,\hat{\omega}^{2}, \,\hat{\omega}\in\mathfrak{so}(3),
\]
where $\boldsymbol{\exp} : \mathfrak{so}(3) \to \SO3$ is the matrix exponential. In particular, for $R_1,R_2\in\SO3$, we have the heat kernel~\cite{nikolayev1997normal}:
\begin{align}
k_{t}(R_1,R_2) = \sum_{\ell=0}^{\infty}(2\ell + 1)e^{-\frac{\ell(\ell + 1)\sigma^2 t}{2}}\frac{\sin\!\bigl((\ell + \tfrac12)\,\|\omega_{12}\|_2\bigr)} {\sin(\|\omega_{12}\|_2/2)}
\label{eq:heat_kernel_SO3}
\end{align}
where $\|\omega_{12}\|_2$ is the magnitude of the rotation angle about the Euler axis between $R_1$ and $R_2$, i.e., $\omega_{12}\in\Real^3$ is such that $R_2^{\top} R_1 = \boldsymbol{\exp}(\hat{\omega}_{12})$, $\hat{\omega}_{12}\in\mathfrak{so}(3)$. We use $\sigma = 0.5$, and the series \eqref{eq:heat_kernel_SO3} truncated at $l_{\text{max}} = 60$ to implement \eqref{eq:(T_t f)(R)}, and thereby the dynamic Sinkhorn recursion in Corollary \ref{corollary:dynamicSinkhorn}.

Using the above, we solve an instance of Problem \ref{problem:SBP_on_G} with endpoint von Mises PDFs on $\SO3$ \cite{lee2013stochastic}:
\begin{equation}
\label{eq:PDF_SO3_vonMises} 
\rho_i(R)
    \propto
\exp\!\left\{\kappa\,\cos(\|\omega\|_2 - \|\omega_i\|_2)
        \right\}, \;i\in\{0,1\},
\end{equation}
where $\omega, \omega_i \in \Real^3$ are related to $R, R_i \in \SO3$ by $R = \boldsymbol{\exp}(\hat{\omega})$ and $R_i = \boldsymbol{\exp}(\hat{\omega}_i)$. For simulation, we fix $\kappa = 30, \|\omega_0\|_2 = 1$, $\|\omega_1\|_2 = 2$. Fig.~\ref{fig:PDF_rotAngle} shows the computed $\rho^{\mathrm{opt}}(\cdot,t), t \in (0, 1)$, supported on $\|\omega\|_2$, with given $\rho_0$ in \emph{dashed black} and $\rho_1$ in \emph{dashed red}, as well as the corresponding Sinkhorn recursion convergence w.r.t. $d_{H}$.

\begin{figure}
    \centering
    \includegraphics[width =0.9\columnwidth]{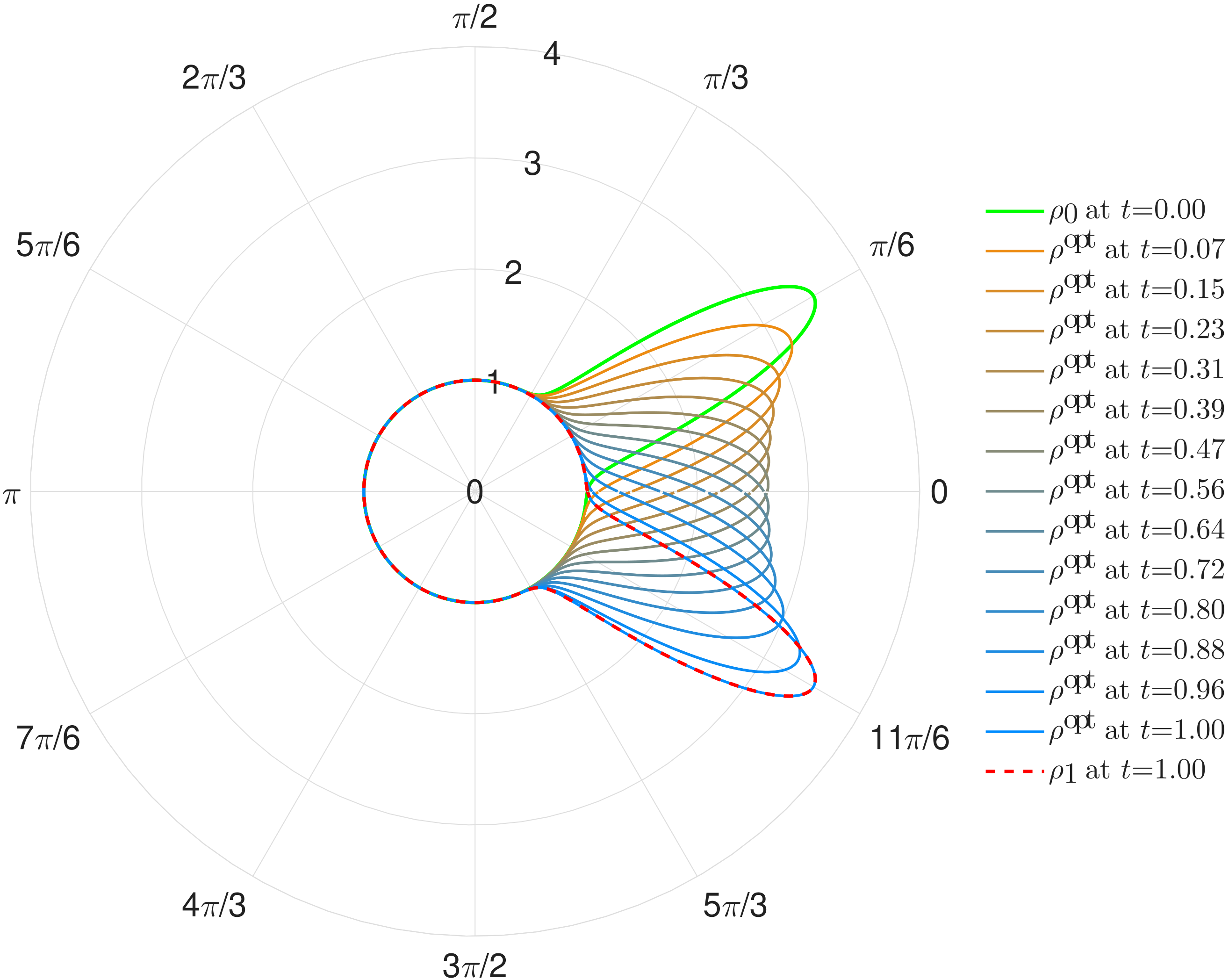}
    \caption{PDFs $\rho^{\mathrm{opt}}(\cdot,t)$, $t\in[0,1]$, for numerical example in Sec. \ref{subsec:SO2numericalexample}.}
    \label{fig:SBP-nearby-peaks}
\end{figure}
\begin{figure}
    \centering
    \includegraphics[width = 0.9\columnwidth]{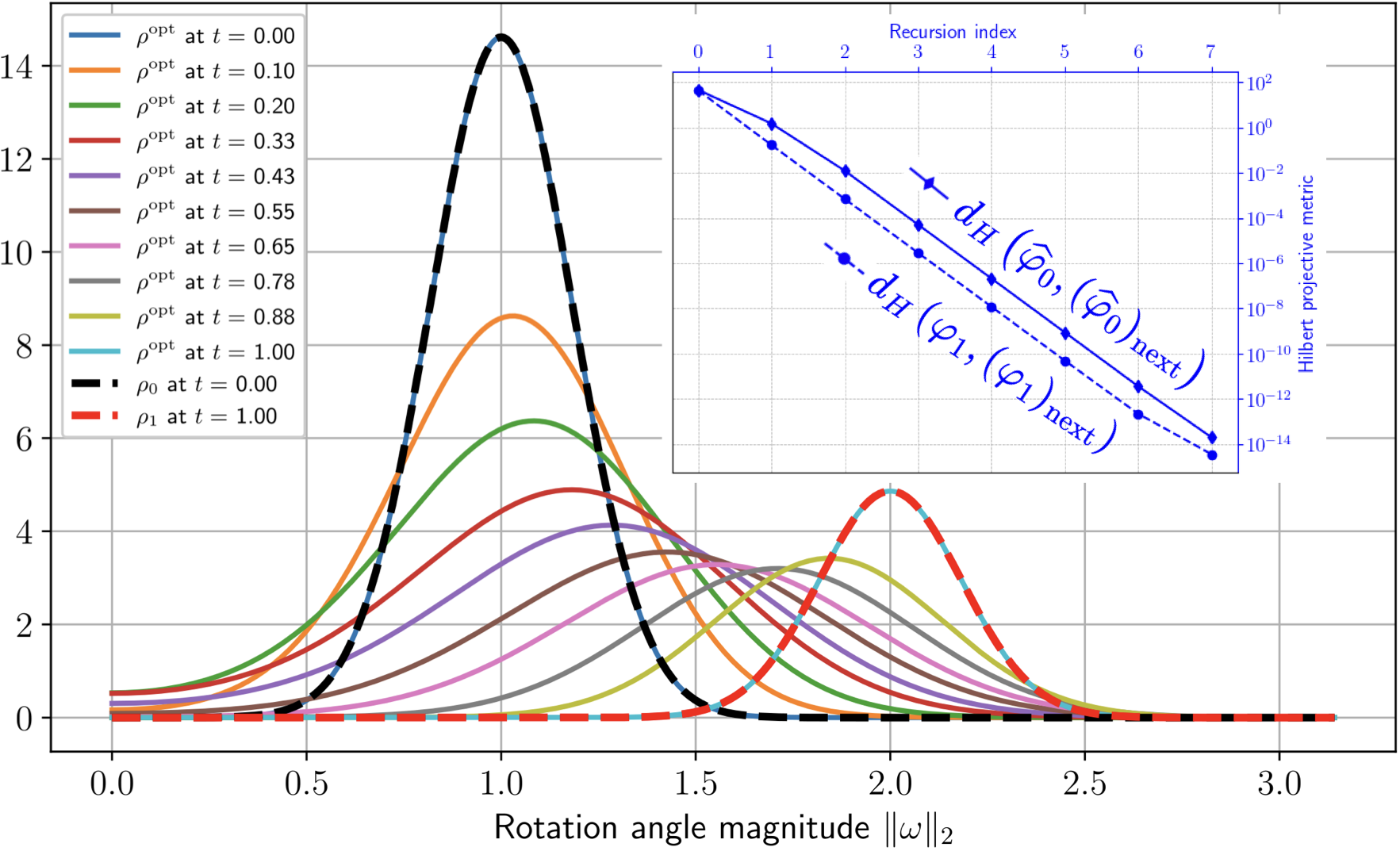}
    \caption{Main plot: PDFs $\rho^{\mathrm{opt}}(\cdot,t)$, $t\in[0,1]$ for numerical example in Sec. \ref{subsec:SO3numericalexample}. Inset: Sinkhorn recursion convergence for pair $(\varphi_1,\widehat{\varphi}_0)$ w.r.t. $d_{H}$.}
    \label{fig:PDF_rotAngle}
    \vspace{-0.15in}
\end{figure}

\section{Concluding Remarks}
\label{sec:Conclusion} 
We study the Schrödinger bridge problem (SBP) on a compact connected Lie group from a stochastic control viewpoint. We propose to leverage the intrinsic geometric structure of the Lie group to formulate, analyze and compute the solution for the SBP. Specifically, we show that the existence-uniqueness of solution, as well as its computation via Sinkhorn recursion, can be done from a geometric (coordinate-free) perspective. Numerical results are given to demonstrate our SBP solutions on $\SO2$ and on $\SO3$. While these two groups were used for illustration, similar numerics can be performed via Peter-Weyl expansion \cite{fegan1983fundamental} of the heat kernel that is valid for any compact Lie group. Indeed, \eqref{eq:heat_kernel_SO3} is a special case of this expansion.

{\myblack
\section*{Appendix}


{\myblack
\begin{proof}[Proof of Lemma~\ref{lem:intC{geq 0}}]
    We first prove $\interior C_{\geq 0}(\mathsf{G}) \subset C_{+}(\mathsf{G})$. Consider any $f \in \interior C_{\geq 0}(\mathsf{G})$. Then there exists $r > 0$ such that the ball $B(f,r)$ centered at $f$ of radius $r$ is contained in $C_{\geq 0}(\mathsf{G})$, i.e., $B(f,r) \subset C_{\geq 0}(\mathsf{G})$. Let $u : \mathsf{G} \to \mathbb{R}$ be the constant function $u(R) = 1$ $\forall R \in \mathsf{G}$. Construct $g : \mathsf{G} \to \mathbb{R}$ as $g \eqdef f - \tfrac{r}{2} u$. Since $f,u \in C(\mathsf{G})$ and $C(\mathsf{G})$ is a real Banach space (and so, a real vector space), we have $g \in C(\mathsf{G})$. Also, it follows from the definition of $g$ that $\|g-f\|_{\infty} < r$, so $g \in B(f,r)$. Since $B(f,r) \subset C_{\geq 0}(\mathsf{G})$, we have $g \in C_{\geq 0}(\mathsf{G})$. This means $\forall R \in \mathsf{G}$, 
    \begin{equation*}
    g(R) \geq 0 \,\, \Rightarrow \,\, f(R) - \frac{r}{2}\,u(R) \geq 0 \,\, \Rightarrow \,\, f(R) \geq \frac{r}{2} > 0.  
    \end{equation*}
    This shows $f \in C_{+}(\mathsf{G})$ and so, $\interior C_{\geq 0}(\mathsf{G}) \subset C_{+}(\mathsf{G})$.  

    We now prove $C_{+}(\mathsf{G}) \subset \interior C_{\geq 0}(\mathsf{G})$. Consider any $h \in C_{+}(\mathsf{G})$. Since $\mathsf{G}$ is a compact Lie group and hence a compact topological space, $h$ being continuous on $\mathsf{G}$ implies that $h$ attains its minimum (because of the compactness of $\mathsf{G}$), that is, there exists an $R^{*} \in \mathsf{G}$ such that 
    $
    \min_{R \in \mathsf{G}} h(R) = h(R^{*}) > 0.  
    $
    We have $h(R^{*}) > 0$ because $h \in C_{+}(\mathsf{G})$. Let $\delta = h(R^{*})$. We now show that $B(h,\delta) \subset C_{\geq 0}(\mathsf{G})$. (This will prove $h \in \interior C_{\geq 0}(\mathsf{G})$.) Consider any $\tilde{h} \in B(h, \delta)$. Then $\|\tilde{h} - h \|_{\infty} < \delta$, which means $\sup_{R \in \mathsf{G}} |\tilde{h}(R) - h(R)| < \delta$. This implies that for each $R \in \mathsf{G}$, 
    \begin{align*}
    &|\tilde{h}(R) - h(R)| \leq \sup_{Q \in \mathsf{G}} |\tilde{h}(Q) - h(Q)| < \delta \\
    \Rightarrow \,\,\, &|\tilde{h}(R) - h(R)| < \delta \\
    \Leftrightarrow \,\,\, &-\delta < \tilde{h}(R) - h(R) < \delta \\
    \Rightarrow \,\,\, &h(R) - \delta < \tilde{h}(R) \\
    \Rightarrow \,\,\, &0 \leq  h(R) - \min_{Q \in \mathsf{G}} h(Q) < \tilde{h}(R), 
    \end{align*}
    where in the last implication, we used $\delta = h(R^{*}) = \min_{Q \in \mathsf{G}} h(Q)$. So, we have for each $R \in \mathsf{G}$, $0 < \tilde{h}(R) \Rightarrow 0 \leq \tilde{h}(R)$. Hence, $\tilde{h} \in C_{\geq 0}(\mathsf{G})$, and so, consequently, we have $B(h, \delta) \subset C_{\geq 0}(\mathsf{G})$. This proves $h \in \interior C_{\geq 0}(\mathsf{G})$, and therefore, $C_{+}(\mathsf{G}) \subset \interior C_{\geq 0}(\mathsf{G})$. So, with both inclusions shown to hold, the proof is complete. 
\end{proof}


\begin{proof}[Proof of Lemma~\ref{lem:R_f_isometry}]
{\myblack 
Since multiplication by a fixed $f \in C_{+}(\mathsf{G})$ is an isometry w.r.t. $d_H$, and $M(\mathbf{x}_1,\mathbf{x}_2) = (m(\mathbf{x}_1^{-1},\mathbf{x}_2^{-1} ))^{-1}$, the statement follows.
}
\end{proof}

\begin{proof}[Proof of Theorem~\ref{thm:Soln_Schrödinger_sys_G}]
{\mydarkred From \eqref{eq:Schrödinger_sys_G}, we have  }
\begin{equation}
\label{eq:varphi_eq}
\varphi_1 \;=\; \frac{\rho_1}{T_1\left(\rho_0/(T_1(\varphi_1))\right)},
\end{equation}
which motivates map $F :  C_{+}(\mathsf{G}) \to C_{+}(\mathsf{G})$ given by 
\begin{equation}
\label{eq:map F}
    F \eqdef R_{\rho_{1}} \circ T_{1} \circ R_{\rho_{0}} \circ T_{1} ,
\end{equation}
wherein $R_{\rho_{0}},R_{\rho_{1}}$ are defined as in Lemma~\ref{lem:R_f_isometry}. Then $\forall f_1, f_2 \in C_{+}(\mathsf{G})$, Lemma~\ref{lem:R_f_isometry} and Lemma~\ref{lem:T_1_str_contraction} yield    
\begin{align}
    &d_H(F(f_1),F(f_2)) \notag \\ 
    &= d_H\Big(\,R_{\rho_{1}} \big(\,T_{1} ( R_{\rho_{0}} ( T_{1}
(f_1)))\,\big) \, , \, R_{\rho_{1}}\big(\,T_{1} ( R_{\rho_{0}} ( T_{1}
(f_2)))\,\big)\,\Big) \notag \\
    &\leq c^2\,\,d_H(f_1, f_2), \;\text{where}\;c \in [0,1)\;\text{from Lemma~\ref{lem:T_1_str_contraction}}.
    \label{F_strict_contraction}
\end{align}
Since $c^2 \in [0,1)$, map $F$ is contractive on $C_{+}(\mathsf{G})$ w.r.t. $d_H$\,. 

Letting $U$ be the unit sphere in the Banach space $\left(C(\mathsf{G}), \|\cdot\|_{\infty} \right)$, consider a map $\tilde{F} : C_{+}(\mathsf{G}) \cap U \to C_{+}(\mathsf{G}) \cap U$ as $\tilde{F}(\phi) \eqdef F(\phi)/\| F(\phi) \|_{\infty}$. Observe that for every $f_1, f_2 \in  C_{+}(\mathsf{G}) \cap U $, we have    
\begin{align*}
d_H\left(\tilde{F}(f_1), \tilde{F}(f_2)\right) &= d_H\left(\frac{F(f_1)}{\| F(f_1) \|_{\infty}} , \frac{F(f_2)}{\| F(f_2) \|_{\infty}}\right) \\ 
&\overset{\mathrm{Remark}~\ref{rem:d_H_properties}}{=} d_H\left(F(f_1), F(f_2)\right) \\ 
&\overset{\mathrm{by\,}\eqref{F_strict_contraction}}{\leq} c^2\,\,d_H(f_1, f_2).  
\end{align*}
With $c^2 \in [0,1)$, {\mydarkred we see} that $\tilde{F}$ is also a strict contraction on its domain $C_{+}(\mathsf{G})\, \cap\, U$ w.r.t. $d_H$\,.

For any $\phi_0 \in C_{+}(\mathsf{G})$, now let $\phi_k \eqdef \tilde{F}^{k}\left(\frac{\phi_0}{\| \phi_0 \|_{\infty}}\right), \,k \geq 1$, where $\tilde{F}^{k} = \tilde{F} \circ \tilde{F} \circ ... \circ \tilde{F}$ (composed $k$ times). Since $\tilde{F}$ is, w.r.t. $d_H$\,, a strict contraction on $C_{+}(\mathsf{G})\, \cap\, U$ which is a complete metric space (Lemma~\ref{lem:Complete_metric_space}), by the Banach fixed point theorem, the map $\tilde{F}$ has a unique fixed point in $C_{+}(\mathsf{G})\, \cap\, U$, i.e., $\exists!\phi^{*} \in C_{+}(\mathsf{G})\, \cap\, U$ such that $\tilde{F}(\phi^{*}) = \phi^{*}$\,. This $\phi^{*}$ is the limit of the sequence $\{\phi_{k}\}_{k \geq 1}$ in $\left(C_{+}(\mathsf{G})\, \cap\, U\,,\,d_H\right)$, i.e., $\phi_k \xrightarrow[]{d_H} \phi^{*}$ in $C_{+}(\mathsf{G})\, \cap\, U$\,, which means that $\lim_{k \to \infty} d_H(\phi_k,\phi^{*}) = 0$.






We now show that $F(\phi^{*}) = \phi^{*}$\,. Since $\tilde{F}(\phi^{*}) = \phi^{*}$\,, we {\mydarkred have} $F(\phi^{*}) = \lambda\, \phi^{*}$\,, where $\lambda := \| F(\phi^{*}) \|_{\infty} > 0$. Define inner product $\langle\cdot,\cdot\rangle_{C(\mathsf{G})} : C(\mathsf{G}) \times C(\mathsf{G}) \to \mathbb{R}$ as $\langle f,g\rangle_{C(\mathsf{G})} \eqdef \int_{\mathsf{G}} f(R)\,g(R)\,\differential\mu(R)$. Since $\langle \phi, R_{\rho_{1}}(\phi) \rangle_{C(\mathsf{G})}$ $= \int_{\mathsf{G}} \rho_1(R)\differential\mu(R) = 1 \;\; \forall\phi \in C_{+}(\mathsf{G})$, we obtain 
\begin{align*}
1 &= \int_{\mathsf{G}} \rho_1 \differential\mu(R)\\  
  &= \langle \,\,(T_1 \circ R_{\rho_0} \circ T_1)(\phi^{*}) \,\,,\, (\underbrace{R_{\rho_1} \circ T_1 \circ R_{\rho_0} \circ T_1}_{F})(\phi^{*}) \,\,\rangle_{C(\mathsf{G})} \\
  &= \langle \,\,(T_1 \circ R_{\rho_0} \circ T_1)(\phi^{*}) \,\,,\, \lambda\, \phi^{*}\,\,\rangle_{C(\mathsf{G})} \\
  &=\langle\,\, (R_{\rho_0} \circ T_1)(\phi^{*}) \,\,,\,T_1(\lambda \phi^{*}) \,\,\rangle_{C(\mathsf{G})} \\
  &=\int_{\mathsf{G}} \rho_0(R)\, \lambda\:\differential\mu(R) = \lambda\,,     
\end{align*}
where the third equality used $F(\phi^{*}) = \lambda\, \phi^{*}$. Hence $F(\phi^{*}) = \lambda\, \phi^{*} = \phi^{*}$\,, and we have  
\begin{align}
\label{eq:sol_varphi}
 \phi^{*} &= F(\phi^{*}) = \rho_1/\left(T_1\left(\rho_0/\left(T_1(\phi^{*})\right)\right)\right).  
\end{align} Comparing \eqref{eq:sol_varphi} with \eqref{eq:varphi_eq}, we see that $\phi^{*}$ solves \eqref{eq:varphi_eq}, i.e.,
\begin{equation}
\label{eq:Schrödinger_sys_soln} \qquad
\varphi_1 \;=\; \phi^{*},\qquad
\widehat{\varphi}_0 = \rho_0/\left(T_1\phi^{*}\right),
\end{equation} 
is a solution of the Schrödinger system \eqref{eq:Schrödinger_sys_G}. The uniqueness of this solution up to multiplication of $\varphi_1$ and division of $\widehat{\varphi}_0$ by some positive constant follows from the uniqueness of $\phi^{*}$. This completes the proof.  
\end{proof}

}

\bibliographystyle{IEEEtran}
\bibliography{myreferences}

\begin{thebibliography}{10}
\providecommand{\url}[1]{#1}
\csname url@samestyle\endcsname
\providecommand{\newblock}{\relax}
\providecommand{\bibinfo}[2]{#2}
\providecommand{\BIBentrySTDinterwordspacing}{\spaceskip=0pt\relax}
\providecommand{\BIBentryALTinterwordstretchfactor}{4}
\providecommand{\BIBentryALTinterwordspacing}{\spaceskip=\fontdimen2\font plus
\BIBentryALTinterwordstretchfactor\fontdimen3\font minus \fontdimen4\font\relax}
\providecommand{\BIBforeignlanguage}[2]{{%
\expandafter\ifx\csname l@#1\endcsname\relax
\typeout{** WARNING: IEEEtran.bst: No hyphenation pattern has been}%
\typeout{** loaded for the language `#1'. Using the pattern for}%
\typeout{** the default language instead.}%
\else
\language=\csname l@#1\endcsname
\fi
#2}}
\providecommand{\BIBdecl}{\relax}
\BIBdecl

\bibitem{chen2021stochastic}
Y.~Chen, T.~T. Georgiou, and M.~Pavon, ``Stochastic control liaisons: Richard sinkhorn meets {G}aspard {M}onge on a {S}chr{ö}dinger bridge,'' \emph{Siam Review}, vol.~63, no.~2, pp. 249--313, 2021.

\bibitem{CalHal2022}
K.~F. Caluya and A.~Halder, ``Wasserstein proximal algorithms for the {S}chrödinger bridge problem: Density control with nonlinear drift,'' \emph{IEEE Transactions on Automatic Control}, vol.~67, no.~3, pp. 1163--1178, 2022.

\bibitem{CheGeoPav2016}
Y.~Chen, T.~T. Georgiou, and M.~Pavon, ``On the relation between optimal transport and {S}chr{\"o}dinger bridges: A stochastic control viewpoint,'' \emph{Journal of Optimization Theory and Applications}, vol. 169, no.~2, pp. 671--691, 2016.

\bibitem{CalHal2019}
K.~F. Caluya and A.~Halder, ``Finite horizon density control for static state feedback linearizable systems,'' \emph{arXiv preprint arXiv:1904.02272}, 2019.

\bibitem{CalHal2020}
K.~F. Caluya and A.~Halder, ``Finite horizon density steering for multi-input state feedback linearizable systems,'' in \emph{2020 American Control Conference (ACC)}.\hskip 1em plus 0.5em minus 0.4em\relax IEEE, 2020, pp. 3577--3582.

\bibitem{CalHal_RefSchBri2020}
K.~F. Caluya and A.~Halder, ``Reflected {S}chr{\"o}dinger bridge: Density control with path constraints,'' in \emph{2021 American Control Conference (ACC)}.\hskip 1em plus 0.5em minus 0.4em\relax IEEE, 2021, pp. 1137--1142.

\bibitem{teter2025hopf}
A.~Teter and A.~Halder, ``On the {H}opf-{C}ole transform for control-affine {S}chrödinger bridge,'' \emph{arXiv preprint arXiv:2503.17640}, 2025.

\bibitem{teter2025schrodinger}
A.~M. Teter, W.~Wang, and A.~Halder, ``Schrödinger bridge with quadratic state cost is exactly solvable,'' \emph{IEEE Transactions on Automatic Control}, 2025.

\bibitem{haddad2020prediction}
S.~Haddad, K.~F. Caluya, A.~Halder, and B.~Singh, ``Prediction and optimal feedback steering of probability density functions for safe automated driving,'' \emph{IEEE Control Systems Letters}, vol.~5, no.~6, pp. 2168--2173, 2020.

\bibitem{nodozi2023neural}
I.~Nodozi, C.~Yan, M.~Khare, A.~Halder, and A.~Mesbah, ``Neural {S}chrödinger bridge with {S}inkhorn losses: Application to data-driven minimum effort control of colloidal self-assembly,'' \emph{IEEE Transactions on Control Systems Technology}, vol.~32, no.~3, pp. 960--973, 2023.

\bibitem{teter2025probabilistic}
A.~M. Teter, I.~Nodozi, and A.~Halder, ``Probabilistic {L}ambert problem: Connections with optimal mass transport, {S}chr{ö}dinger bridge, and reaction-diffusion {PDE}s,'' \emph{SIAM Journal on Applied Dynamical Systems}, vol.~24, no.~1, pp. 16--43, 2025.

\bibitem{Bir1957}
G.~Birkhoff, ``Extensions of {J}entzsch's theorem,'' \emph{Transactions of the American Mathematical Society}, vol.~85, no.~1, pp. 219--227, 1957.

\bibitem{Bus1973}
P.~J. Bushell, ``Hilbert's metric and positive contraction mappings in a {B}anach space,'' \emph{Archive for Rational Mechanics and Analysis}, vol.~52, no.~4, pp. 330--338, 1973.

\bibitem{lee2013stochastic}
T.~Lee, ``Stochastic optimal motion planning and estimation for the attitude kinematics on {SO(3)},'' in \emph{52nd IEEE Conference on Decision and Control}.\hskip 1em plus 0.5em minus 0.4em\relax IEEE, 2013, pp. 588--593.

\bibitem{chirikjian2011}
G.~S. Chirikjian, \emph{Stochastic models, information theory, and {L}ie groups, volume 2: Analytic methods and modern applications}.\hskip 1em plus 0.5em minus 0.4em\relax Springer Science \& Business Media, 2011, vol.~2.

\bibitem{Piotr2007}
P.~Koszmider, ``Chapter 52 - the interplay between compact spaces and the {B}anach spaces of their continuous functions,'' in \emph{Open Problems in Topology II}, E.~Pearl, Ed.\hskip 1em plus 0.5em minus 0.4em\relax Amsterdam: Elsevier, 2007, pp. 567--580.

\bibitem{tapp2016matrix}
K.~Tapp, \emph{Matrix groups for undergraduates}.\hskip 1em plus 0.5em minus 0.4em\relax American Mathematical Soc., 2016, vol.~79.

\bibitem{folland1999real}
G.~B. Folland, \emph{Real analysis: modern techniques and their applications}.\hskip 1em plus 0.5em minus 0.4em\relax John Wiley \& Sons, 1999.

\bibitem{Faraut2008}
J.~Faraut, \emph{Analysis on {L}ie Groups: An Introduction}, ser. Cambridge Studies in Advanced Mathematics.\hskip 1em plus 0.5em minus 0.4em\relax Cambridge: Cambridge University Press, 2008.

\bibitem{bullo2005geometric}
F.~Bullo and A.~D. Lewis, \emph{Geometric control of mechanical systems: modeling, analysis, and design for simple mechanical control systems}.\hskip 1em plus 0.5em minus 0.4em\relax Springer, 2005, vol.~49.

\bibitem{hsu2002}
E.~P. Hsu, \emph{Stochastic analysis on manifolds}.\hskip 1em plus 0.5em minus 0.4em\relax American Mathematical Soc., 2002, no.~38.

\bibitem{Lee2025}
\BIBentryALTinterwordspacing
T.~Lee and G.~S. Chirikjian, ``Geometric interpretation of {B}rownian motion on {R}iemannian manifolds,'' 2025. [Online]. Available: \url{https://arxiv.org/abs/2510.19991}
\BIBentrySTDinterwordspacing

\bibitem{Grig2012}
A.~Grigor'yan, \emph{Heat Kernel and Analysis on Manifolds}.\hskip 1em plus 0.5em minus 0.4em\relax American Mathematical Society, 2012.

\bibitem{Che2016}
Y.~Chen, T.~Georgiou, and M.~Pavon, ``Entropic and displacement interpolation: a computational approach using the {H}ilbert metric,'' \emph{SIAM Journal on Applied Mathematics}, vol.~76, no.~6, pp. 2375--2396, 2016.

\bibitem{peyre2019computational}
G.~Peyr{\'e} and M.~Cuturi, \emph{Computational optimal transport: with applications to data science}.\hskip 1em plus 0.5em minus 0.4em\relax Now Foundations and Trends, 2019.

\bibitem{stein2011fourier}
E.~M. Stein and R.~Shakarchi, \emph{Fourier analysis: an introduction}.\hskip 1em plus 0.5em minus 0.4em\relax Princeton University Press, 2011, vol.~1.

\bibitem{nikolayev1997normal}
D.~I. Nikolayev and T.~I. Savyolov, ``Normal distribution on the rotation group {SO(3)},'' \emph{Texture, Stress, and Microstructure}, vol.~29, no. 3-4, pp. 201--233, 1997.

\bibitem{fegan1983fundamental}
H.~Fegan, ``The fundamental solution of the heat equation on a compact {L}ie group,'' \emph{Journal of differential geometry}, vol.~18, no.~4, pp. 659--668, 1983.

\end{thebibliography}

\end{document}